\documentclass[11pt]{amsart}
\usepackage{amsfonts,amssymb,amsthm,eucal,mathtools}

\usepackage[margin=1.1in]{geometry}
\usepackage{url, enumitem, color, cite}
\usepackage{hyperref}

\newtheorem{thm}{Theorem}[section]
\newtheorem{cor}[thm]{Corollary}
\newtheorem{rk}[thm]{Remark}

\newtheorem{lemma}[thm]{Lemma}
\newtheorem{prop}[thm]{Proposition}
\newtheorem{defn}[thm]{Definition}

\newcommand{\R}{\mathbb{R}}
\newcommand{\E}{\mathbb{E}}
\newcommand{\M}{\mathbb{M}}
\newcommand{\p}{\mathbb{P}}
\newcommand{\N}{\mathbb{N}}
\newcommand{\Z}{\mathbb{Z}}
\newcommand{\C}{\mathbb{C}}

\newcommand{\I}[1]{{\bf 1}_{#1}}
\DeclareMathOperator{\tr}{tr}
\DeclareMathOperator{\rank}{rank}

\DeclareMathOperator{\essup}{ess\ sup}
\DeclarePairedDelimiter\ceil{\lceil}{\rceil}

\newcommand{\Prob}[1]{\mathbb{P}\left( #1 \right)}
\newcommand{\abs}[1]{\left\vert #1 \right\vert}
\newcommand{\norm}[1]{\left\Vert #1 \right\Vert}

\allowdisplaybreaks[4]

\author{Bart{\l}omiej \ Polaczyk}

\address{Institute of Mathematics,
University of Warsaw}

\email{polaczyk.b@gmail.com}

\title[Concentration of the spectral distribution 
	of random matrices]
	{Concentration of the empirical spectral distribution 
	of random matrices with dependent entries}

\subjclass[2010]{60B20} 
\keywords{random matrix theory, 
 empirical spectral distribution, 
 concentration of measure}
\thanks{
 Research partially supported by the National Science Centre, 
 Poland, grants no. 2015/18/E/ST1/00214}

\begin{document}

\begin{abstract}
 We investigate concentration properties of spectral measures of 
 Hermitian random matrices with partially dependent entries.
 More precisely,
 let $X_n$ be a Hermitian random matrix of size $n\times n$ that can be 
 split into independent blocks of the size at most $d_n=o(n^2)$. 
 We prove that under some mild conditions on the distribution
 of the entries of $X_n$, the empirical spectral measure of 
 $X_n$ concentrates around its mean. 
 
 The main theorem is a strengthening of a recent result by
 Kemp and Zimmerman, where the size of blocks grows as $o(\log n)$.
 As an application, we are able to upgrade the results of 
 Schenker and Schulz on the convergence in expectation to 
 the semicircle law of a class of random matrices with dependent 
 entries to weak convergence in probability. 
 Other applications include patterned random matrices, 
 e.g. matrices of Toeplitz, Hankel or circulant type
 and matrices with heavy tailed entries in the domain of
 attraction of the Gaussian distribution. 
\end{abstract}

\maketitle


\addtocontents{toc}{\protect\setcounter{tocdepth}{1}}
\tableofcontents

\section{Introduction} 

Throughout this paper we will denote by $M_n$ the space of 
$n\times n$ matrices over the scalar field $\mathbb{C}$
equipped with the Hilbert-Schmidt norm 
$\norm{A}_{HS}:=\sqrt{\tr AA^\ast}$, 
where $(A^\ast)_{ij} = \overline{A_{ji}}$.
We set $M_n^{sa}$ to be the vector subspace of $M_n$ 
consisting of Hermitian matrices 
(i.e. matrices satisfying the condition $A^*=A$). 
A (general) random matrix is 
a random variable taking values in the space $M_n$.

Let $X$ be a random $n\times n$ Hermitian matrix. Its all eigenvalues
lie on the real line and thus we may consider its 
empirical spectral distribution (ESD) being equal to
\[
 L_n^X := \frac{1}{n}\sum_{i=1}^{n}\delta_{\lambda_i},
\]
where $\lambda_1\leq\ldots\leq\lambda_n$ are the eigenvalues of $X$.
It is worth to remark that since $X$ is random then so is $L_n^X$ as a 
distribution on the real line. One can thus consider its 
expected value, which is now a deterministic probability measure 
$\E L_n^X$ s.t. for every Borel set $A$,
\[
 \E L_n^X (A) := \E (L_n^X(A)).
\]

Studying the asymptotic properties of such distributions
was firstly motivated by questions that arose in various models of
quantum physics (c.f. \cite{Wig,BIPZ,t'H}).
Since then, random matrix theory has evolved significantly, 
becoming an independent and influential branch of mathematics
(we refer to \cite{AGZ,BaiSilv,PaSh,Tao}
for a	 detailed exposition of random matrix theory).
While the first results in random matrix theory considered matrices 
with independent entries (up to a symmetry condition), over the last 
20 years more and more attention has been directed also to investigation 
of matrices with dependencies between entries
(see e.g. \cite{Ada11,BHS11,BDJ,Mec09,O'Rour,Zhou}).

The paper will be organized in the following way. 
We begin with presenting the results in Section~\ref{S:MainRes}. 
Then, in Section~\ref{S:Examples} we will show some
applications, which will cover
Wigner-type theorems, matrices with heavy tailed entries in the 
domain of attraction of the Gaussian distribution, patterned and band matrices.
Section~\ref{S:AuxLemm} contains some facts from linear algebra
and concentration of measure theory that will be used in Section~\ref{S:Proofs}
to prove our results.

\section{Main results}\label{S:MainRes}
In this paper we restrict our attention to the matrices with the block dependency
structure, i.e. matrices whose entries can be divided into blocks
which form independent random vectors. The following definition makes
this notion precise.

\begin{defn}
 We say that a random matrix $X$ of the size $n\times m$ 
 satisfies the property $\mathcal{S}(d)$ whenever there exists a partition 
 $\Pi = \{ P_1,\ldots,P_{k}\}$ of the set 
 $\{1,\ldots,n\}\times\{1,\ldots,m\}$ such that
 the vectors $Y_r=\{X_{ij}\}_{(i,j)\in P_r}$ are stochastically independent
 and the size of each partition set $P_r$ 
 does not exceed~$d$, i.e. $\abs{P_r}\leq d$ for all $r=1,\ldots,k$.
 We write shortly $X\in\mathcal{S}(d)$.
\end{defn}

The starting point of our considerations is the main theorem
of \cite{KZ}, stated below.

\begin{thm}\label{T:KZ}
 Let $X_n\in M_n^{sa}$ be a sequence of random matrices such that
 $X_n\in\mathcal{S}(d_n)$ for every~$n$ with $d_n=o(\log n)$.
 If the family 
 $\{ |(X_n)_{ij}|^2\}_{1\leq i,j \leq n\in\N}$ is uniformly integrable,
 then 
 \begin{equation*}\label{eq:Kemp}
  \forall_{f\in C_L(\R)} ~
  \int f dL_n^{\frac{1}{\sqrt{n}}X_n} 
  - \E\int f dL_n^{\frac{1}{\sqrt{n}}X_n} 
  \to_{\mathbb{P}} 0,
 \end{equation*}
 with $C_L(\R)$ denoting the set of all real 1-Lipschitz functions on $\R$
 and $(X_n)_{ij}$ denoting the entries of the matrix $X_n$.
\end{thm}

The proof is based on the 
concentration argument by Guionnet and Zeitouni (\cite{GZ})
and log-Sobolev inequalities for compactly supported measures
convolved with the standard Gaussian distribution derived by 
Kemp and Zimmerman.

The main result of this work may be seen as a stronger version 
of Theorem \ref{T:KZ} where more dependency is allowed, i.e. $d_n=o(n^2)$.
Before stating it, let us clarify some notation.
 
\begin{defn}
 We say that a sequence of probability measures $\mu_n$ converges weakly
 to some measure $\mu$ if $\int f\,d\mu_n \to \int f\,d\mu$ for
 all continuous bounded functions $f$.

 We say that a sequence of random probability measures 
 $\mu_n$ converges weakly in probability to the (possibly random) measure $\mu$ if
 $d(\mu_n,\mu) \to_{\mathbb{P}} 0$ 
 for some (equivalently for all) metric $d$ that metrizes
 the above notion of weak convergence. 
 
 We denote these facts
 by $\mu_n\Rightarrow\mu$ and $\mu_n\Rightarrow_{\mathbb{P}}\mu$ respectively.
\end{defn}

\begin{defn}
 We say that a sequence of random matrices $X_n\in M_n$ 
 has the property $\mathcal{L}$, $(X_n)_{n\in\N}\in\mathcal{L}$, if it satisfies the
 following Lindeberg-type condition
 \begin{equation*}\label{eq:SCondition}
  \forall_{\epsilon>0}~\
  \lim_{M\to \infty} \limsup_{n\to \infty} \Prob{
  \frac{1}{n^2}\sum_{i,j=1}^n |(X_n)_{ij}|^2 \I{\{|(X_n)_{ij}| > M\}} 
  > \varepsilon} =0.
 \end{equation*} 
\end{defn}

The main result of this work is the following theorem.

\begin{thm}\label{T:main}
 Let $X_{n}\in M_n^{sa}$ be a sequence of random matrices such that
 $X_n\in\mathcal{S}(d_n)$ for every~$n$ with $d_n=o(n^2)$.
 If $(X_{n})_{n\in\N} \in \mathcal{L}$,
 then for any metric $d$ that metrizes weak convergence of probability measures,
 \begin{equation*}\label{eq:MainThesis}
  d(L^{\frac{1}{\sqrt{n}}X_n}_n,\E L^{\frac{1}{\sqrt{n}}X_n}_n) \to_{\mathbb{P}} 0.
 \end{equation*}
 In particular, if $\E L_n^{\frac{1}{\sqrt{n}}X_n}\Rightarrow\mu$,
 then $L_n^{\frac{1}{\sqrt{n}}X_n}\Rightarrow_{\mathbb{P}}\mu$.
\end{thm}

The following observations show connection between Theorems \ref{T:KZ} and \ref{T:main}. 

\begin{rk}\label{Rk:UI}
 It follows easily from Markov's inequality that if the family 
 $\lbrace \abs{(X_n)_{ij}}^2\rbrace_{1\leq i,j\leq n\in\N}$ is uniformly integrable,
 then $(X_n)_{n\in \N}\in\mathcal{L}$.
\end{rk}

\begin{prop}\label{P:Tight}
 If $(X_n)_{n\in\N}\in\mathcal{L}$, 
 then the sequence $\E L_n^{\frac{1}{\sqrt{n}}X_n}$ is tight.
\end{prop}

\begin{prop}\label{P:Reduction}
 If the sequence $\E L_n^{\frac{1}{\sqrt{n}}X_n}$ is tight, 
 then the following conditions are equivalent:
 \begin{enumerate}[label=\roman*)]
  \item \label{eq:CondMetr}
  $
  d(L_n^{\frac{1}{\sqrt{n}}X_n},\E L_n^{\frac{1}{\sqrt{n}}X_n})\to_{\mathbb{P}}0
  $
  for any metric $d$ that metrizes weak convergence of probability measures,
  \item 
 $\label{eq:CondFuncC_b}
  \forall_{f\in C_b(\R)} \quad
  \int fdL^{\frac{1}{\sqrt{n}}X_n}_n - 
  \E\int fdL^{\frac{1}{\sqrt{n}}X_n}_n \to_{\mathbb{P}} 0,
 $
 \item
 $\label{eq:CondFuncC_cC_L}
  \forall_{f\in C_c(\R)\cap C_L(R)} \quad
  \int fdL^{\frac{1}{\sqrt{n}}X_n}_n - 
  \E\int fdL^{\frac{1}{\sqrt{n}}X_n}_n \to_{\mathbb{P}} 0,
 $
 \end{enumerate}
 where $C_b(\R)$ and $C_c(\R)$ denote the sets of all 
 bounded (resp. compactly supported) real continuous functions on $\R$.

 If additionally the family 
 $\lbrace \abs{(X_n)_{ij}}^2\rbrace_{1\leq i,j\leq n\in\N}$
 is uniformly integrable, then all the above conditions become 
 equivalent to
 \begin{enumerate}[label=\roman*),resume]
 \item
 $\label{eq:CondFuncC_L}
  \forall_{f\in C_L(\R)} \quad
  \int fdL^{\frac{1}{\sqrt{n}}X_n}_n - 
  \E\int fdL^{\frac{1}{\sqrt{n}}X_n}_n \to_{\mathbb{P}} 0.
 $
 \end{enumerate}
\end{prop}
%
 Combining the above observations asserts us that
 Theorem \ref{T:main} strengthens Theorem \ref{T:KZ}.
 Moreover, it can be easily seen that Theorem~\ref{T:main} is optimal in terms of
 the blocks size.
 
\begin{rk}\label{Rk:Loosening}
 The assumption $d_n=o(n^2)$ in Theorem \ref{T:main}
 is optimal for the convergence in probability. 
 To see that, consider two random matrix ensembles $X_n$, $Y_n\in M_{n}$, 
 whose ESDs converge a.s. to distinct limits $\mu$ and $\nu$.
 Set
 \[
  Z_n = 
  \begin{bmatrix}
   \epsilon X_{nt_n} + (1-\epsilon)Y_{nt_n} & 0\\
   0 & I_{n(1-t_n)}
  \end{bmatrix},
 \]
 where $I_k\in M_k$ is the identity matrix, 
 $\epsilon$ is a Bernoulli variable 
 $(\Prob{\epsilon=1} = \Prob{\epsilon=0}=1/2)$
 independent of all $X_n$ and $Y_n$ and
 $t_n\in (0,1)$ is a sequence converging to some $t\in(0,1)$,
 s.t. $nt_n$ and $n(1-t_n)$ are integers for every $n$.
 One can see that $d_n = (t_nn)^2$,
 $L_n^{Z_n} \Rightarrow t(\epsilon\mu + (1-\epsilon)\nu) + (1-t)\delta_{1}$
 a.s. and thus
 $d(L_n^{Z_n},\E L_n^{Z_n})$
 cannot converge in probability to zero for any metric $d$
 that metrizes weak convergence of probability measures.
\end{rk}

Using standard hermitization technique one can extend the result of
Theorem \ref{T:main} onto the convergence of the distribution of 
singular values of not necessarily Hermitian ensembles.

\begin{thm}\label{T:2nd}
 Let $X_{n}$ be a sequence of $n\times N$ random matrices (with $N=N(n)$) 
 such that $X_n\in\mathcal{S}(d_n)$ for every~$n$ with $d_n=o(n^2)$
 and set $Y_{n} = \sqrt{X_{n}X_{n}^\ast}$. 
 Assume $n/N \to c \in (0,\infty)$.
 If the family $\{X_{n}\}_{n\in\N}$ 
 satisfies the Lindeberg-type condition 
 \[
  \forall_{\varepsilon > 0} ~\
  \lim_{M\to \infty} \limsup_{n\to \infty} \Prob{
  \frac{1}{nN}\sum_{i=1}^n\sum_{j=1}^{N} 
  |(X_n)_{ij}|^2 \I{\{|(X_n)_{ij}|^2 > M\}} 
  > \varepsilon} = 0,
 \] 
 then for any metric $d$ that metrizes weak convergence of probability measures
 \begin{equation*}
  d(L^{\frac{1}{\sqrt{n}}Y_n}_n, \E L^{\frac{1}{\sqrt{n}}Y_n}_n) \to_{\mathbb{P}} 0.
 \end{equation*}
\end{thm}

 Finally, Theorem~\ref{T:main} will be deduced from 
 another, more general result, stated below.
 
\begin{defn}
 We say that a sequence of random matrices $X_n\in M_n$ 
 has the property $\mathcal{L}(a_n)$, $(X_n)_{n\in\N}\in\mathcal{L}(a_n)$,
 for some sequence $a_n\to\infty$ if it satisfies the
 following Lindeberg-type condition
 \begin{equation*}\label{eq:LTcond2}
  \forall_{\epsilon>0}~\
  \limsup_{n\to \infty} \Prob{
  \frac{1}{n^2}\sum_{i,j=1}^n 
   |(X_n)_{ij}|^2 \I{\{|(X_n)_{ij}| > \varepsilon a_n\}} 
  > \varepsilon} =0.
 \end{equation*} 
\end{defn}
 
\begin{thm}\label{T:generalization}
 Let $X_{n}\in M_n^{sa}$ be a sequence of random matrices such that
 $X_n\in\mathcal{S}(d_n)$ for every~$n$ with $d_n=O(n^2/a_n^2)$ for some sequence
 $a_n\to\infty$.
 If $(X_{n})_{n\in\N} \in \mathcal{L}(a_n)$,
 then
 \begin{equation*}
  \forall_{f\in C_c(\R)\cap C_L(\R)} \quad
  \int fdL^{\frac{1}{\sqrt{n}}X_n}_n - 
  \E\int fdL^{\frac{1}{\sqrt{n}}X_n}_n \to_{\mathbb{P}} 0.
 \end{equation*}
\end{thm}


\section{Consequences and examples}\label{S:Examples}

\subsection{Wigner-type theorems}
The first result in the theory of asymptotic behavior of the spectrum of
random matrices (and the work that can be considered the starting point of 
random matrix theory) goes back to the paper \cite{Wig}.
It states that if $(X_n)$ is a sequence of real Hermitian random matrices 
with i.i.d. entries (up to the symmetry constraint) with zero mean and 
variance equal to one, then $L_n^{\frac{1}{\sqrt{n}}X_n}$ converges 
almost surely to the semicircular distribution $\sigma$, that is
\[
 \Prob{L_n^{\frac{1}{\sqrt{n}}X_n}\Rightarrow\sigma} = 1,
\]
where
\[
 \sigma(x) = \frac{1}{2\pi}\sqrt{4-x^2}\I{\abs{x}\leq 2}
\]
is the Wigner semicircular distribution playing the analogous
role in free probability as plays the Gaussian distribution in
classical probability.

Recently Schenker and Schulz-Baldes (c.f. \cite{SchSch}) proved 
a version of Wigner Theorem in which one allows some degree of 
dependence between the entries of a matrix and the price paid is 
the weaker notion of convergence obtained, i.e. convergence in
expectation instead of probability.

More precisely, let $\sim_n$ denote an equivalence relation on 
$\lbrace 1,2,\ldots,n \rbrace^{2}=:[n]^{2}$ 
and let $X_n$ 
be a sequence of Hermitian random matrices s.t. 
random vectors made of entries of $X_n$ belonging to
distinct equivalence classes are independent (and the dependence
between elements of the same class can be arbitrary).
We impose conditions on $X_n$ and $\sim_n$ as follows
\begin{enumerate}
 \item[(C0)]
  $\forall_{k\in\N} \
  \sup\{\E \abs{(X_n)_{ij}}^k: \ 1\leq i,j\leq n\in \N \} < \infty$,
 \item[(C1)]
  $\max_{i\in[n]} \#{
   \lbrace (j,i',j')\in[n]^{ 3}: \,
    (i,j)\sim_n (i',j')
   \rbrace} = o(n^2)$,
 \item[(C2)]
  $\exists_{B>0} 
   \forall_{n\in\N} \
   \max_{(i,j,i')\in[n]^{ 3}} \#{
   \lbrace j'\in[n]: \,
    (i,j)\sim_n (i',j')
   \rbrace} \leq B$,
 \item[(C3)] 
  $\#{\lbrace (i,j,i')\in[n]^{ 3}: \,
    (i,j)\sim_n (j,i'), ~\ i'\neq i
   \rbrace} = o(n^2)$.
\end{enumerate}

Now, the main theorem of \cite{SchSch} can be stated as follows.

\begin{thm}\label{T:Sch}
 If a sequence $(X_n,\sim_n)_{n\in N}$ satisfies conditions (C0)--(C3), 
 then
 \[
  \E L_n^{\frac{1}{\sqrt{n}}X_n} \Rightarrow \sigma.
 \]
\end{thm}

The above result was highly motivated by applications, in particular
by the analysis of the Anderson model.
The standard Anderson model is given by the
following random Hamiltonian acting on the space $\ell^2(\Z^d)$
\[
 H\psi (x) =
 \sum_{\abs{y-x}=1}\psi(y) + V(x)\psi(x),
\]
where $\{V(x)\}_{x\in \mathbb{Z}^d}$ is a family of standard Gaussian
i.i.d. random variables.
In \cite{Poi}, it was shown that the above model at small disorder 
can be analyzed by dividing the space into small cubes $\Lambda$.
If $V_{\Lambda}$ is the restriction of $H$ to one such cube $\Lambda$,
then  it can be effectively approximated by
a finite random matrix whose coefficients are
centered complex Gaussian random variables with a given dependency structure
(see \cite{Bel,SchSch} for more details).
Finding the limiting spectral distribution of such matrices 
for $d=2$ was solved in \cite{BMR}. The case $d\geq 3$ was dealt with in
\cite{SchSch}, where the authors showed that $V_{\Lambda}$ falls into
the regime of Theorem~\ref{T:Sch} and thus the limiting distribution
of $V_{\Lambda}$ under appropriate normalization is semicircular.

Note that the condition $d_n=o(n^2)$ from Theorem \ref{T:main} 
can be written in the above language as 
 $$
 \max_{(i,j)\in [n]^2} 
 \#\lbrace{
  (i',j')\in [n]^2: \, (i,j)\sim_n(i',j')
 \rbrace} = o(n^2),
 $$ 
which is clearly implied by condition (C1).
Moreover, condition (C0) implies uniform
integrability of $\lbrace \abs{(X_n)_{ij}}^2\rbrace_{1\leq i,j\leq n\in\N}$,
which by Remark~\ref{Rk:UI} implies that $(X_n)_{n\in\N}\in\mathcal{L}$.
Theorem~\ref{T:main} gives therefore the following strengthening of 
Theorem~\ref{T:Sch}.

\begin{cor}
 If a sequence $(X_n,\sim_n)_{n\in N}$ satisfies conditions (C0)--(C3), then
 \begin{equation*}
  L_n^{\frac{1}{\sqrt{n}}X_n} \Rightarrow_{\mathbb{P}} \sigma.
 \end{equation*}
\end{cor}

The above may be seen as a special case of application of Theorem \ref{T:main},
which in general allows (whenever the assumptions are met) to strengthen 
convergence in expectation to convergence in probability. 
The very same scheme may be applied to strengthen 
Theorem 5.1 from \cite{HoSto}
where the authors develop further the method of Schenker and Schultz to
deal with matrices of the form
\begin{equation*}\label{eq:Chiral_class}
 A = 
 \begin{bmatrix}
  0 & X\\
  X^* & 0
 \end{bmatrix}
 \text{ for } X\in \C^{s\times t},
\end{equation*}
where the limiting measure is the Marchenko-Pastur distribution.

\subsection{Matrices with heavy tailed entries}
Let us recall that a mean zero random variable $x$ is in the domain of attraction of the
Gaussian distribution if there exists a sequence $b_n$ s.t.
\[
 \mathcal{L}\left(\frac{\sum^n_{i=1} x_i}{b_n}\right) \Rightarrow \mathcal{N}(0,1),
\] 
where $x_i$'s are i.i.d. copies of $x$. 
Here $\mathcal{L}(x)$ denotes the law of the r.v. $x$.
It can be shown that this is the case if and only if the function
\[
 l(t)=\E x^2\I{\{\abs{x}\leq t\}}
\] is slowly varying at infinity (see e.g. \cite{AraGin,IbLin}).

Assume that $x$ has infinite variance, define
\begin{align}\label{D:b_n}
 b := \inf\{t>0:\ l(t)>0\},
 \quad
 b_n := \inf\{t>b+1:\ nl(t)\leq t^2\}
\end{align}
and consider a matrix $X_n\in M_n^{sa}$, whose entries are i.i.d. copies of $x$.
It was shown in \cite{Zhou}
(see also \cite{Ada13}) that $L_n^{X_n/b_n}\Rightarrow\sigma$ almost surely.
The results of this paper allow us to prove the convergence in probability 
in case of dependent entries, yielding the following proposition.

\begin{prop}\label{P:HeavyTails}
 Let $X_n\in M_n^{sa}$ be a sequence of random matrices satisfying 
 $X_n\in\mathcal{S}(d_n)$ with $d_n = O(n)$, 
 whose entries have the same distribution as 
 a random variable $x$ with zero mean, infinite variance and in the domain of
 attraction of the Gaussian distribution. 
 Then for any metric $d$ that metrizes weak convergence of probability measures
 \[
  d(L_n^{X_n/b_n},\E L_n^{X_n/b_n}) \to_{\mathbb{P}} 0.
 \]
\end{prop}

The proof is moved to the last section. 
In the above we can observe a drop in the size of blocks $d_n$ compared to
Theorem~\ref{T:main}. It is not straightforward to see if this result
can be improved.

\subsection{Patterned matrices}

Many ensembles of random matrices considered in the literature can
be seen as a special cases of the so called patterned matrices.
Following \cite{BHS}, let us consider a family of functions
$\mathcal{G}=\{I_n:\{1,\ldots,n\}^2\to \Z^d\}_{n\in\N}$,
which we will call a link family. A patterned matrix $X_n$ is a matrix
of the form $\overline{(X_n)_{ji}} = (X_n)_{ij} = [Z_{I_n(i,j)}]$ for $i\leq j$
where $\mathcal{Z} = \{Z_z\}_{z\in\Z^d}$ 
is a family of independent random variables 
(note that by construction we demand that
$X_n\in M_n^{sa}$, which gives some constrains 
on $\mathcal{G}$ and $\mathcal{Z}$).
We say that the sequence $X_n$ 
is associated with the link family $\mathcal{G}$.
Theorems \ref{T:main} and \ref{T:2nd} 
yield the following corollary.

\begin{cor}\label{C:Patterned}
 If a sequence $\{X_n\}_{n\in\N}$ associated with $\mathcal{G}$ 
 belongs to the class $\mathcal{L}$, $(X_n)_{n\in\N}\in \mathcal{L}$,
 and $\abs{I_n^{-1}(z)} = o(n^2)$ for every $z\in \Z^d$, then
 for any metric $d$ that metrizes weak convergence of probability measures
 \begin{equation*}
  d(L_n^{\frac{1}{\sqrt{n}}X_n}, \E L_n^{\frac{1}{\sqrt{n}}X_n})
  \to_{\mathbb{P}} 0.
 \end{equation*}
\end{cor}

Setting $I_n(i,j) = (\min(i,j),\max(i,j))$ 
restores the generic Wigner ensemble.

Setting $I_n(i,j) = \abs{i-j}$ or $I_n(i,j) = i+j$ 
with $\mathcal{Z}$ being an i.i.d. family
results in Toeplitz and Hankel ensembles respectively, 
considered firstly in the influential paper \cite{Bai}.
The problem of the convergence of their ESDs
remained unsolved until the work \cite{BDJ} appeared, where the
authors prove even the almost sure convergence
of $T_n/\sqrt{n}$ and $H_n/\sqrt{n}$ to some deterministic distributions
that do not depend on the law of $Z_0$ and have unbounded support.

Setting $I_n(i,j) = i+j \pmod{n}$ or 
$I_n(i,j) = n/2 - \abs{n/2-\abs{i-j}}$ results in
reversed circulant and symmetric circulant ensembles respectively.
These ensembles (and more general $G$-circulants) were
extensively studied 
(c.f. \cite{Ada17,BHS,Mec09,Mec12}).

It can be easily seen that 
under some mild assumptions on the family $\mathcal{Z}$
all these ensembles satisfy hypothesis
of Corollary \ref{C:Patterned}.
Moreover, whenever all elements of $\mathcal{Z}$ are with zero mean,
infinite variance and in the domain of
attraction of the Gaussian distribution, then all these ensembles
satisfy assumptions of Proposition~\ref{P:HeavyTails} as well.

Additionally, now Theorem \ref{T:main} allows us to simplify the proofs
from \cite{Ada17} and \cite{Mec12} where the authors strengthen
convergence in expectation to convergence in probability
(c.f. the proof of Theorem~4.1 in \cite{Mec12}
 and the proof of Theorem~1.5 in \cite{Ada17}).

\subsection{Band and block matrices with correlation structure}

A classical band matrix ensemble consists of matrices whose entries are 
independent and equal zero at far distance from the diagonal.
It is known (c.f. \cite{AnZe}) that if one assumes some regularity
of the distribution of the entries, then the ESD of such matrices converges
almost surely to the semicircular measure.

There are few papers however covering the behavior of the ESD of
band matrices with dependent entries. Some of the best known results can be
found in \cite{Sly} and \cite{ROBS}.
The last paper deals with a wider class of block matrices.
Such ensembles arise naturally in applications, e.g. in
wireless communication theory in the Multiple-Input Multiple-Output 
(MIMO) systems with Intersymbolic Interference (ISI).
The capacity of such system, with $n$ transmit antennas 
and $m$ receive antennas, can be described in terms of the ESD of the matrix
$GG^\ast$, where $G$ is a band random matrix consisting of the 
finite number of matrices $(A_l)$ of the size $n\times m$. 
The elements of $A_l$ are independent and the correlation structure 
between $A_i$ and $A_j$ is given 
(for a more precise formulation we refer to \cite{LaSto}, Chapter 2).
In \cite{ROBS}, the authors have proved (in the Gaussian case), 
that the ESD of (appropriately normalized) $GG^\ast$
converges almost surely, as $n,m$ tend to infinity, to some deterministic
probability measure described in terms of its Cauchy transform.
Using Theorem~\ref{T:2nd}, one immediately deduces that the empirical 
spectral measure is concentrated around its expectation, which in 
combination with their analysis of the expected spectral measure gives a
weaker property of weak convergence in probability. 
However, as can be easily seen from 
the proof of Theorem~\ref{T:generalization}, 
in this case our argument gives in fact the almost sure 
sure convergence, since the size of independent blocks remains bounded.


\section{Auxiliary lemmas, facts and definitions}\label{S:AuxLemm}
We start with recalling some definitions and important results. 

\begin{defn}
 We say that a random vector $X$ satisfies the (subgaussian) concentration
 property with positive constants $C$ and $c$ with respect to the family
 $\mathcal{F}$ if
 \begin{equation}\label{D:CP}
 \Prob{ \abs{f(X)-\M f(X)} > t }
 \leq
 Ce^{-t^2/c}
 \end{equation}
 for all $t\geq 0$ and for every every $f\in\mathcal{F}$
 (with $\M f$ denoting the median of $f$). 
\end{defn}

\begin{rk}\label{R:MeanMedian}
 A standard observation is that in \eqref{D:CP} one can replace the median 
 by mean (at the cost of enlarging of $c$ by some multiplicative factor
 $\widetilde{c}$ depending only on $C$).
\end{rk}

Substituting $\mathcal{F} = C_L(\R)$ 
restores the definition of the classical subgaussian concentration property
to which we will refer simply as $CP(C,c)$ and substituting
$$
\mathcal{F} = \lbrace f\in C_L(\R): \, f \text{ is convex} \rbrace
$$
restores the weaker notion of the convex concentration property 
(as stated in e.g. \cite{MS}), to which we will refer as $CCP(C,c)$.


%

Some convex concentration results concerning
the spectral distribution of random matrices were firstly discovered 
by Guionnet and Zeitouni (c.f. \cite{GZ}). 
The proofs are mostly based on the
famous theorem due to Talagrand (c.f. \cite{Ta1,Led})
whose corollary we state below.

\begin{thm}[\cite{Mec02}, Corollary 4]
\label{T:Talagrand}
 Let $V$ be the direct sum of the normed vector spaces 
 $(V_i,\norm{\cdot}_i)_{1\leq i\leq N}$, equipped with the norm 
 $\norm{(v_1,\ldots,v_N)} = \sqrt{\sum \norm{v_i}_i^2}$. 
 Let $X_i$ be a random variable taking values in $V_i$ s.t.
 $\essup \norm{X_i}_i \leq \rho$. 
 Then the $V$-valued random vector $(X_1,\ldots,X_N)$ satisfies 
 CCP$(4,16\rho^2)$.
\end{thm}

The following is the so-called Hoffman--Wielandt Lemma 
and its immediate corollary
(recall that $C_L(\R)$ denotes the set of all 1-Lipschitz functions on $\R$).
\begin{lemma}[\cite{HW}, Theorem 1]\label{L:HW}
 Let $A,B\in M_{n}^{sa}$ with eigenvalues $\lambda_1^A\leq \ldots 
 \leq \lambda_n^A$ and $\lambda_1^B\leq \ldots \leq \lambda_n^B$
 respectively. Then
 \[
   \sum_{i=1}^{n}\abs{\lambda_i^A-\lambda_i^B}^2 \leq \|A-B\|_{HS}^{2}.
 \]
\end{lemma}

 
\begin{cor} \label{C:HW}
 For every $f\in C_L(\R)$ the map 
 $M_n^{sa}(\C) \ni X\to \int f\,dL_n^{\frac{1}{\sqrt{n}}X}$ is 
 $\frac{1}{n}$-Lipschitz with respect to the 
 Hilbert-Schmidt norm.
\end{cor}

%
%
%

We will also need the following classical observation (sometimes called
Klein's lemma). For a proof we refer to 
\cite{GZ}, Lemma 1.2 or \cite{MOA}, Theorem 9.G.1.

\begin{lemma}\label{L:Klein}
 If $f$ is a real valued convex function on $\R$, then the mapping
 $M_n^{sa}(\C)\ni X\to\int fd\,L_n^{X}$ is convex.
\end{lemma}

As a consequence of the above
observations, we obtain that whenever the entries 
of a random matrix $X\in M_n^{sa}$ have compact support, then
$X_n$ satisfies convex concentration property 
(w.r.t. the Hilbert-Schmidt norm) which will allow us to estimate 
$\int f\,dL_n-\E\int f\,dL_n$ for big $n$.
This fact will play a crucial role in the proof of 
Theorem~\ref{T:generalization}.
A similar argument was used to prove convex concentration
in \cite{GuLe} (c.f. proof of Thm.~6).
For other results concerning convex concentration of random matrices
see e.g. \cite{Mec02,Del}.

Finally, the following facts will be important for
the proof of Proposition~\ref{P:HeavyTails}

 \begin{lemma}[\cite{BaiSilv}, Theorem A.43]\label{L:rankIneq}
  Let $X$ and $Y$ be two $n\times n$ Hermitian matrices. Then
  \[
   \norm{L_n^{A} - L_n^B} \leq \frac{1}{n}\rank(A-B),
  \]
  where $\norm{\cdot}$ denotes the Kolomogorov distance between measures.
 \end{lemma}

 \begin{lemma}[\cite{Zha}, proof of Theorem~3.32]\label{L:momentEst}
 Let $X\in M_n^{sa}$ with rows $x_1,\ldots,x_n$ and eigenvalues 
 $\lambda_1,\ldots,\lambda_n$. Then for every $0< r \leq 2$
 \[
  \sum_{i=1}^{n}\abs{\lambda_i}^r \leq \sum_{i=1}^{n}\norm{x_i}_2^r.
 \]
 \end{lemma}


\section{Proofs}\label{S:Proofs}

\subsection*{Proofs of Propositions \ref{P:Tight} and \ref{P:Reduction}}

In what follows, we will keep the notation 
$L_n = L_n^{\frac{1}{\sqrt{n}}X_n}$. 
\begin{proof}[Proof of Proposition \ref{P:Tight}]
 Fix $\varepsilon > 0$, let $C=[-K,K]^c$ for some $K>0$ and set
 \[
  A_n =  \Big\{	
    \frac{1}{n^2}\sum_{i,j=1}^n |(X_n)_{ij}|^2 \I{\{|(X_n)_{ij}|^2 > M\}}
    < 1
   \Big\}.
 \]
 Then
 \begin{align*}
  \E L_n(C) &=
  \E L_n(C)\I{A_n^c} + \E L_n(C)\I{A_n}\\
  &\leq
  \Prob{A_n^c} + \E\frac{\int x^2\,dL_n}{K^2}\I{A_n}
  \leq
  \Prob{A_n^c} + \frac{M+1}{K^2},
 \end{align*}
 where the first inequality is an application of Markov's inequality
 and the second uses the fact that
 $\int x^2\,dL_n = \frac{1}{n^2}\norm{X_n}_{HS}^2$.
 Since $X_n\in\mathcal{L}$, 
 choosing appropriate $M$ provides that 
 $\Prob{A_n^c} < \varepsilon/2$ for $n$ big enough.
 Taking then $K$ big enough yields $E L_n(C) < \varepsilon$
 for every $n$ and the result follows.
\end{proof}


\begin{proof}[Proof of Proposition \ref{P:Reduction}]

 Recall that a sequence of random elements with values in a 
 metric space converges in probability to some random element 
 if and only if from each of its subsequences one can choose 
 a further subsequence that converges almost surely to that element.
 Moreover, by Prokhorov's theorem, 
 a family of measures on a Polish space is tight if and only if it
 is sequentially weakly compact (c.f. \cite{Kal}, Lemma~3.2, Prop.~4.21).

 $\ref{eq:CondMetr}\Rightarrow \ref{eq:CondFuncC_b}.$ 
 Take any sequence $N\subset\N$. We will find  
 a further subsequence along which $\int f\,d L_n - \E\int f\, dL_n$
 converges a.s. to zero for any $f\in C_b(\R)$.
 By tightness, there exist $N'\subset N$ and a probability measure 
 $\mu$ s.t. $\E L_{n}\Rightarrow \mu$ along $N'$.
 By the triangle inequality, $d(L_n, \mu)\to_{\mathbb{P}} 0$ along $N'$ and thus
 there exists a further subsequence $N''\subset N'$ s.t.
 $d(L_n, \mu)\to 0$ a.s. along $N''$, which yields the result.

 Implication $\ref{eq:CondFuncC_b}\Rightarrow \ref{eq:CondFuncC_cC_L}$ is trivial.
  
 $\ref{eq:CondFuncC_cC_L}\Rightarrow \ref{eq:CondMetr}.$
 Consider the metric $d$ on the set of all probability measures on $\R$ given by
 \[
  d(\mu,\nu) = \sum_{k\in\N}2^{-n}\abs{\int f_k\,d\mu - \int f_k\,d\nu},
 \]
 where $\{f_k\}_{k\in\N}\subset C_c(\R)$ 
 is some dense subset of the unit ball (in the $\sup$ norm) of $C_c(\R)$. 
 It is easy to see that $d$ metrizes weak convergence of probability measures. 
 We have
 \begin{align*}
  \Prob{d(L_n,\E L_n) > \varepsilon}
  &\leq
  \Prob{ \sum_{k=1}^{N}2^{-k}
   \abs{\int f_k\,dL_n - \E\int f_k\,dL_n} > \frac{\varepsilon}{2}}\\
  &\leq
  \sum_{k=1}^{N}
  \Prob{\abs{\int f_k\,dL_n - \E\int f_k\,dL_n} >\frac{2^{k-1}\varepsilon}{N}},
 \end{align*}  
 for some $N$ big enough, depending only on $\varepsilon$.
 Now, for every $k$, there exists a constant $L_k$ 
 and a compactly supported $L_k$-Lipschitz function $g_k$ s.t.
 $\norm{f_k-g_k}_{\infty}<\frac{2^{k-1}\varepsilon}{3N}$, whence using the triangle
 and \ref{eq:CondFuncC_cC_L} yields \ref{eq:CondMetr}.

 To prove the second part of Proposition \ref{P:Reduction},
 assume firstly conditions $\ref{eq:CondMetr} -\ref{eq:CondFuncC_cC_L}$, 
 fix $t,\varepsilon>0$ and consider a function $f\in C_L$. 
 Let $f_r$ be a continuous function, 
 equal $f$ on the interval $[-r,r]$ and constant beyond it. 
 Since $f_r\in C_b(\R)$, then by assumption
 $\int f_r\,dL_n-\E\int f_r\,dL_n\to_{\mathbb{P}} 0$.
 We have
 \[
 \E\int \abs{f-f_r}\,dL_n
 \leq
 \E\int \abs{x}\I{\{\abs{x}>r\}}\, dL_n
 \leq
 \frac{1}{r}\E\int \abs{x}^2\, dL_n
 =
 \frac{1}{r}\E\frac{1}{n^2}\norm{X_n}_{HS}^2 \leq \delta
 \]
 for any $\delta$ for $r$ big enough by the uniform integrability of 
 $\{\abs{(X_n)_{ij}}^2\}_{1\leq i,j \leq n\in\N}$. 
 Take now $\delta \leq \min(\frac{t}{3},\frac{t\varepsilon}{6})$
 and choose $r$ such that the above estimate holds.
 Applying now the triangle and Markov's inequalities yields
 \begin{align*}
  \Prob{\abs{\int f\,dL_n-\E\int f\,dL_n} > t}
  &\leq
  \Prob{\abs{\int f\,dL_n-\int f_r\,dL_n} > \frac{t}{3}}\\
  &+
  \Prob{\abs{\int f_r\,dL_n-\E\int f_r\,dL_n} > \frac{t}{3}}
  \leq
  \varepsilon
 \end{align*}
 for $n$ big enough.
 The proof in the opposite direction is immediate.
\end{proof}

\subsection*{Proofs of Theorems \ref{T:main}, \ref{T:2nd} and \ref{T:generalization}}

We will start with proving Theorem~\ref{T:generalization}.
The argument is highly motivated by the work of Guionnet and Zeitouni 
(c.f. \cite{GZ}, Theorem 1.3). The conclusions of Theorem~\ref{T:main} 
and \ref{T:2nd} will then follow easily.
\begin{proof}[Proofs of Theorem \ref{T:generalization}]

	Denote $L_n=L_n^{X_n/\sqrt{n}}$ and let $f\in C_c(\R)\cap C_L(\R)$ 
	be supported in the interval $[-M,M]$.	
	By Proposition~\ref{P:Reduction}, it is enough to show that for every 
	$\delta,t>0$ and $n$ big enough,
	$$\Prob{ \abs{\int f\,dL_n - \E \int f\,dL_n} > t } <\delta.$$
	Let $\delta$ and $t$ be fixed from now on.
	Take some $\varepsilon>0$ (to be fixed later), set 
	$$
	 (X_n^{\varepsilon})_{ij} 
	 = (X_n)_{ij}\I{\{\abs{(X_n)_{ij}} \leq \varepsilon a_n \}}.
	$$
	and denote $L_n^{\varepsilon} = L_n^{X_n^{\varepsilon}/\sqrt{n}}$. 
	Firstly, we will show that there exists	$\varepsilon$ s.t.
	$$\Prob{ \abs{\int f\,dL^{\varepsilon}_n 
	- \E \int f\,dL^{\varepsilon}_n} > \frac{t}{3} } 
	<\frac{\delta}{2}$$
	for $n$ big enough and then that 
	$L_n^{\varepsilon}$ and $L_n$ do not differ much.
	The first step will be achieved by means of Theorem~\ref{T:Talagrand}.
		
	Let $\Pi_n = \{ P^n_1,\ldots,P^n_k \}$ be the partition of $X_n$ into
	independent blocks and $Y_{r,n}^{\varepsilon}$ 
	be the random vector given by the entries 
	$\{(X_n^{\varepsilon})_{ij}\}_{(i,j)\in P^n_r}$. 
	Since $Y_{1,n}^{\varepsilon},\ldots,Y_{k,n}^{\varepsilon}$ 
	are stochastically independent and
	\begin{align*}
	 \essup\norm{Y_{r,n}^{\varepsilon}}_{2} 
	 \leq
	 \sqrt{\sum_{(i,j)\in P_r^n}(\essup\abs{(X_n^{\varepsilon})_{ij}})^2}
	 \leq
	 \sqrt{d_n}\varepsilon a_n,
	\end{align*}
	then Theorem~\ref{T:Talagrand} implies that 
	$X^{\varepsilon}_n$ satisfies CCP$(4,16d_n \varepsilon^2 a_n^2)$.
	
	The aim is now to approximate $f$ with a finite 
	combination of convex functions, 
	which will allow us to exploit CCP of $X^{\varepsilon}_n$. 
	To that end, let $\Delta$ be small enough 
	(to be fixed later) and (following \cite{GZ}) set
	\[
	 g(x) := 
	 \begin{cases}
	  0 ~ \text{ for } x\leq 0,\\
	  x ~ \text{ for } 0\leq x \leq \Delta,\\
	  \Delta ~ \text{ otherwise}.
	 \end{cases}
	\]
	Note that $g$ can be decomposed into a difference of two convex functions.
	Define now recursively $g_0 \equiv 0$,
	\[
	 g_{n+1}(x) = 
	 \begin{cases}
	  g_n(x) + g(x+M-n\Delta) 
	   ~ \text{ if } f(-M+(n+1)\Delta) \geq g_n(-M+n\Delta),\\
	  g_n(x) - g(x+M-n\Delta) 
	   ~ \text{ otherwise }
	 \end{cases}
	\]
	and set $f_\Delta := g_{\ceil{2M/\Delta}}$.
	We can see that $\norm{f-f_{\Delta}}\leq \Delta$ and $f_{\Delta}$ can be
	decomposed into a sum of at most $2\ceil{2M/\Delta}=:\kappa$ different convex 
	and concave 1-Lipschitz functions $\{h_l\}$. 
	Set $\Delta < t/6$.
	Exploiting Corollary~\ref{C:HW}, Lemma~\ref{L:Klein} 
	and CCP of $X_n^{\varepsilon}$ 	(c.f. Remark \ref{R:MeanMedian}) results in
	\begin{align*}
	 \Prob{ \abs{\int f\,dL^{\varepsilon}_n - 
	  \E \int f\,dL^{\varepsilon}_n} > \frac{t}{3} }
	 &\leq
	 \Prob{ \abs{\int f_{\Delta}\,dL^{\varepsilon}_n - 
	  \E \int f_{\Delta}\,dL^{\varepsilon}_n} > \frac{t-6\Delta}{3} }\\
	 & \leq
	 \kappa \sup_{l} \Prob{\abs{\int h_l\,dL^{\varepsilon}_n - 
	  \E\int h_l\,dL^{\varepsilon}_n} > \frac{t-6\Delta}{3\kappa}}\\
	 & \leq
	 4\kappa \exp\left(
	  -\frac{(t-6\Delta)^2}
	   {144\widetilde{c}\kappa^2}\frac{n^2}{d_na_n^2 \varepsilon^2 } \right).
	\end{align*}
	Fix now $\varepsilon$ s.t. for all $n$ the above quantity is smaller than 
	$\frac{\delta}{2}$ (it is possible to do so since $d_n = o(n^2)$).
	Using again Corollary~\ref{C:HW} and
	the condition $(X_n)_{n\in\N}\in\mathcal{L}(a_n)$, we infer that
	\begin{align*}
	 \Prob{\abs{
	  \int f\,dL_n - \int f\,dL_n^{\varepsilon}  } > \frac{t}{3}}
	 &\leq
	 \Prob{\norm{X_n - X^{\varepsilon}_n}_{HS} > \frac{tn}{3}}
	 \\
	 &=
	 \Prob{\frac{1}{n^2}\sum_{i,j=1}^n |(X_n)_{ij}|^2 
	  \I{\{|(X_n)_{ij}| > \varepsilon a_n\}} 
	 > \frac{t^2}{9}}
	 \leq \frac{\delta}{2}
	\end{align*}
	for $n$ big enough.
	Moreover, since the last quantity actually converges to zero with $n$
	(for any $\varepsilon$ and every $t$), we have
	$\int f\,dL_n-\int f\,dL_n^{\varepsilon}\to_{\p} 0$, whence  
	boundedness of	$f$ implies that
	$\E \abs{\int f\,dL_n - \int f\,dL_n^{\varepsilon}} \leq t/3$ 
	for $n$ big enough.
	For $n$ such that all the above estimates hold, the triangle inequality yields
	\begin{align*}
	 \Prob{\abs{
	  \int f\,dL_n
	   - \E \int f\,dL_n
	  } > t} &\leq
	 \Prob{\abs{
	  \int f\,dL_n - \int f\,dL_n^{\varepsilon}
	  } > \frac{t}{3}}\\
	  &+
	 \Prob{\abs{
	  \int f\,dL_n^{\varepsilon}
	   - \E \int f\,dL_n^{\varepsilon}
	  } > \frac{t}{3}} \leq \delta,
	\end{align*}
	which by Proposition~\ref{P:Reduction} concludes the proof.
\end{proof}

\begin{proof}[Proof of Theorem~\ref{T:main}]
	It can be easily checked that $(X_n)_{n\in\N}\in\mathcal{L}$ implies
	$(X_n)_{n\in\N}\in\mathcal{L}(a_n)$ for any sequence $a_n\to\infty$.
	For	$a_n^2 = n^2/d_n$ we have that $d_n=O(n^2/a_n^2)$ and thus $X_n$ satisfies
	the assumptions of Theorem~\ref{T:generalization}.
	Now Propositions~\ref{P:Tight} and \ref{P:Reduction} allow us to conclude
	the proof.
\end{proof}

\begin{proof}[Proof of Theorem~\ref{T:2nd}]
	The proof boils down to the use of the so-called hermitization technique. 
	Consider the matrix
	\[
	 A_n =
	 \begin{bmatrix}
	  0 & X_n \\
	  X_n^\ast & 0
	 \end{bmatrix}.
	\]
	Clearly, the ESD of $Y_n$ can be inferred from the ESD of $A_n$. 
	Moreover, $A_n$
	meets the assumptions of Theorem~\ref{T:main}, which yields the result.
\end{proof}

\subsection*{Proof of Proposition \ref{P:HeavyTails}}

 To prove Proposition \ref{P:HeavyTails}, we will need the following 
 auxiliary fact.
 
 \begin{rk}\label{Rk:HeavyEst}
  If $x$ is in the domain of attraction of the Gaussian distribution,
  $b_n$'s are defined as in \eqref{D:b_n} and 
  $l(t) = \E\abs{x}^2\I{\{\abs{x} <t\}}$, then
 \begin{align*}\label{Eq:Heavy}
  \lim_{n\to\infty}\frac{nl(b_n)}{b_n^2} =1,
  \quad
  \Prob{\abs{x} > t} = o(l(t)/t^2)
  \quad
  \text{and}
  \quad
  \E \abs{x}\I{\{\abs{x}>t\}}= o(l(t)/t).
 \end{align*}
 The first equality follows easily from the definitions of $b_n$ and $l$,
 while to prove the remaining equalities one has to make use of
 the fact that $l$ is slowly varying 
 (c.f.~\cite{Ada13}, Proof of Corollary~2.10).
 \end{rk}

 In what follows, set
 $(\widetilde{X}_n)_{ij} = (X_n)_{ij}\I{\{\abs{(X_n)_{ij}} \leq b_n \}}$.
 
 \begin{lemma}\label{L:HeavyTight}
  If $X_n$'s satisfy the assumptions of Proposition~\ref{P:HeavyTails}, then
  $\E L_n^{{X_n}/{b_n}}$ is tight.
 \end{lemma} 
 \begin{proof}[Proof of Lemma~\ref{L:HeavyTight}]
 Using Lemma~\ref{L:momentEst} with $r=1$
 and denoting rows of $X_n$ by $(X_n)_i$, we arrive at
 \begin{align*}
  \E\int\abs{x}\,dL_n^{X_n/b_n} 
  &\leq
  \frac{1}{nb_n}\E\left[ \sum_{1\leq i\leq n}
   \norm{(\widetilde{X}_n)_i}_2+\norm{(X_n)_i-(\widetilde{X}_n)_i}_2\right].
 \end{align*}
 Applying arithmetic vs quadratic mean and Jensen's inequalities
 together with Remark~\ref{Rk:HeavyEst} yields
 \[
  \frac{1}{nb_n}\E\sum_{1\leq i\leq n}\norm{(\widetilde{X}_n)_i}_2
  \leq
  \frac{1}{\sqrt{n}b_n}\E\norm{\widetilde{X}_n}_{HS}
  \leq
  \frac{1}{\sqrt{n}b_n} \sqrt{\E\sum_{ij} \abs{(\widetilde{X}_n)_{ij}}^2}
  \leq
  \frac{\sqrt{nl(b_n)}}{b_n}
  =
  O(1),
 \]
 whereas the norm inequality $\norm{\cdot}_2\leq\norm{\cdot}_1$
 and Remark~\ref{Rk:HeavyEst} give
 \begin{align*}
  \frac{1}{nb_n}\E\sum_{1\leq i\leq n}\norm{(X_n)_i-(\widetilde{X}_n)_i}_2
  \leq
  \frac{1}{nb_n}\E\sum_{1\leq i\leq n}\norm{(X_n)_i-(\widetilde{X}_n)_i}_1
  &=
  \frac{n}{b_n}\E\abs{x}\I{ \{ \abs{x}>b_n \} } = o(1).
 \end{align*}
 The above estimates give a uniform upper bound on the first moment of 
 $\E L_n^{X_n/b_n}$, whence the conclusion follows.
 \end{proof}
 
 \begin{proof}[Proof of Proposition~\ref{P:HeavyTails}]

 Denote by $L_n$ and $\widetilde{L}_n$ the ESDs of $b_n^{-1}X_n$ and 
 $b_n^{-1}\widetilde{X}_n$ respectively.
 Recall that Kolomogorov's metric defined as the sup distance between
 cumulative distribution functions dominates L\'{e}vy-Prokhorov's metric
 defined as
 \[
  \pi(\mu,\nu)
   =
  \inf\{ \varepsilon >0:~\
   \forall_{t\in\R}~\
    F_\nu(t-\varepsilon)-\varepsilon \leq 
    F_\mu(t)
    \leq
    F_\nu(t+\varepsilon)+\varepsilon   
  \},
 \]
 where $F_\sigma$ denotes the c.d.f. of a measure $\sigma$.
 Note that the latter metrizes weak convergence of probability measures 
 (c.f. \cite{Bill}, Theorem~6.8).
 Now, by Lemma \ref{L:rankIneq} and Remark~\ref{Rk:HeavyEst}
 \begin{align*}
  \E \pi(L_n,\widetilde{L}_n) \leq
  \E\norm{L_n - \widetilde{L}_n} \leq
  \E\frac{1}{n}\rank(X_n-\widetilde{X}_n)
  &\leq
  \E\frac{1}{n}\sum_{1 \leq i,j\leq n} \I{\{ \abs{(X_n)_{ij}} > b_n \}}\\
  &=
  n\Prob{|x| > b_n}=
  o(1)
 \end{align*}
 and thus we reduced the problem to proving the convergence of $\widetilde{L}_n$.
 We will achieve that by showing that 
 $\frac{\sqrt{n}}{b_n}\widetilde{X}_n$ falls into the regime of 
 Theorem~\ref{T:generalization} with $a_n=\sqrt{n}$.
 By Markov's inequality
 \begin{align*}
  \Prob{
  \frac{1}{n^2}\sum_{i,j=1}^n 
   |\frac{\sqrt{n}}{b_n}(\widetilde{X}_n)_{ij}|^2 
   \I{\{\frac{\sqrt{n}}{b_n}|(\widetilde{X}_n)_{ij}| > \varepsilon\sqrt{n}\}} 
   > \varepsilon}
  &=
  \Prob{
  \frac{1}{nb_n^2}\sum_{i,j=1}^n 
   |({X}_n)_{ij}|^2 
   \I{\{b_n \geq |({X}_n)_{ij}| > \varepsilon b_n\}} 
   > \varepsilon}\\
   &\leq
   \frac{nl(b_n)}{\varepsilon b_n^2}\frac{l(b_n)-l(\varepsilon b_n)}{l(b_n)}
   =o(1)
 \end{align*}
 where in the last equality we have used Remark~\ref{Rk:HeavyEst}
 and the fact that $l$ is slowly varying.
 Thus 
 $(\frac{\sqrt{n}}{b_n}\widetilde{X}_n)_{n\in\N}\in\mathcal{L}(a_n)$, whence
 applying Lemma~\ref{L:HeavyTight} and Proposition~\ref{P:Reduction} 
 yields the result.
 \end{proof}


\section{Acknowledgment}

The results presented in this article are a part of a master's 
thesis written at the University of Warsaw under supervision of 
R. Adamczak, who suggested the problem, gave the theoretical tools
to solve it and guided me patiently through the whole process of
creating a publishable paper, for which I am sincerely grateful.

\bibliographystyle{plain}

\end{document}